\documentclass[10pt]{amsart}

\usepackage{comment, amsmath, amssymb, amsgen, amsthm, amscd, xspace, color, epsfig, float, epic, multirow, array, lscape}
\usepackage{extarrows}

\makeatletter
\newcommand{\rmnum}[1]{\romannumeral #1}
\newcommand{\Rmnum}[1]{\expandafter\@slowromancap\romannumeral #1@}

\newcommand{\vf}{\varphi}

\newcommand{\ch}{\operatorname{ch}}

\newcommand{\Ker}{\mathrm{Ker}}

\newcommand{\Soc}{\mathrm{Soc}}
\newcommand{\Rad}{\mathrm{Rad}}

\newcommand{\ra}{\rightarrow}

\newcommand{\frb}{\mathfrak{b}}

\newcommand{\frg}{\mathfrak{g}}
\newcommand{\frh}{\mathfrak{h}}

\newcommand{\frl}{\mathfrak{l}}

\newcommand{\frn}{\mathfrak{n}}

\newcommand{\frp}{\mathfrak{p}}
\newcommand{\frq}{\mathfrak{q}}

\newcommand{\fru}{\mathfrak{u}}

\newcommand{\bbC}{\mathbb{C}}

\newcommand{\bbN}{\mathbb{N}}

\newcommand{\bbZ}{\mathbb{Z}}

\newcommand{\caF}{\mathcal{F}}

\newcommand{\caJ}{\mathcal{J}}
\newcommand{\caK}{\mathcal{K}}

\newcommand{\caO}{\mathcal{O}}

\newcommand{\caT}{\mathcal{T}}

\newtheorem{theorem}[equation]{Theorem}

\newtheorem{prop}[equation]{Proposition}
\newtheorem{lemma}[equation]{Lemma}

\theoremstyle{remark}
\newtheorem{remark}[equation]{Remark}

\theoremstyle{definition}
\newtheorem{definition}[equation]{Definition}

\numberwithin{equation}{section} 
\begin{document}

\title[Jantzen Conjecture]{Jantzen Conjecture for singular characters}

\author{Wei Xiao}
\address[Xiao]{College of Mathematics and statistics, Shenzhen Key Laboratory of Advanced Machine Learning and Applications, Shenzhen University,
Shenzhen, 518060, Guangdong, P. R. China}
\email{xiaow@szu.edu.cn}

\thanks{The authors are supported by the National Science Foundation of China (Grant No. 11701381) and Guangdong Natural Science Foundation (Grant No. 2017A030310138).}




\bigskip

\begin{abstract}
We show that the Jantzen filtration of a Verma module (possibly singular) coincides with its radical filtration. It implies that the Jantzen Conjecture on Verma modules holds for all infinitesimal characters, while the regular case was settled by Beilinson and Bernstein by geometric methods and reproved by Williamson by an algebraic approach. The coincidence between Jantzen filtration and radical filtration is also generalized to the case of parabolic Verma modules by Shan's results.
\end{abstract}


\maketitle

%
%
\section{Introduction}
%
%
Let $\frg$ be a complex semisimple Lie algebra with a fixed Borel subalgebra $\frb$ and a Cartan subalgebra $\frh\subset\frb$. Let $\Phi$ be the root system of $(\frg, \frh)$ with the positive system $\Phi^+$ and simple system $\Delta$ corresponding to $\frb$. The \emph{Verma module} is
\[
M(\lambda):=U(\frg)\otimes_{U(\frb)}\bbC_{\lambda-\rho},
\]
where $\bbC_{\lambda-\rho}$ is a one dimensional $\frb$-module of weight $\lambda-\rho$ and $\rho$ is the half sum of positive roots. Those modules are standard modules in the BGG category $\caO$ \cite{BGG}. Denote $\Phi_\mu^+=\{\alpha\in\Phi^+\mid \langle\mu, \alpha^\vee\rangle\in\bbZ^{>0}\}$. Here $\langle, \rangle$ is the pair between $\frh^*$ and $\frh$ and $\alpha^\vee$ is the \emph{coroot} of $\alpha$ (see \cite[\Rmnum{6}, 1]{B}).

\smallskip
\noindent{\bf Jantzen filtration.} Let $\mu\in\frh^*$. Then $M(\mu)$ has a filtration by submodules
\begin{equation}\label{ieq1}
M(\mu)=M(\mu)^0\supset M(\mu)^1\supset M(\mu)^2\supset\ldots
\end{equation}
with $M(\mu)^i=0$ for large $i$, such that:
\begin{itemize}
\item [(1)] Every quotient $M(\mu)^i/M(\mu)^{i+1}$ has a nondegenerate contravariant form.

\item [(2)] $M(\mu)^1$ is the unique maximal submodule of $M(\mu)$.

\item [(3)] There is a formal character formula:
\begin{equation}\label{ieq2}
\sum_{i>0}\ch M(\mu)^i=\sum_{\alpha\in\Phi_\mu^+}\ch M(s_\alpha\mu).
\end{equation}
\end{itemize}
The equation (\ref{ieq2}) is known as \emph{Jantzen sum formula}.

This remarkable filtration introduced by Jantzen \cite{J2} provides deep information about Verma modules. Since nonzero homomorphisms between Verma modules are always injective, Jantzen raised the following famous question.

\smallskip
\noindent{\bf Jantzen Conjecture.} Suppose that $M(\mu)\subset M(\lambda)$ for $\lambda, \mu\in\frh^*$. Then
\begin{equation}
M(\mu)\cap M(\lambda)^i=M(\mu)^{i-r}
\end{equation}
for $i\geq r$, where $r=|\Phi_\lambda^+|-|\Phi_\mu^+|$.

\smallskip

Understanding the composition factor multiplicites of Verma modules is a core problem in representation theory around 1980s. A precise conjecture was proposed (for regular integral infinitesimal character) by Kazhdan and Lusztig in their 1979 paper \cite{KL}. Soon Gabber and Joseph showed that a weak version of the Jantzen Conjecture lead to the validity of Kazhdan-Lusztig Conjecture \cite{GJ}. This implies that the Jantzen Conjecture may be stronger than the Kazhdan-Lusztig Conjecture. Almost at the same time, the Kazhdan-Lusztig Conjecture was proved independently by Beilinson-Bernstein \cite{BB1} and Brylinsky-Kashiwara \cite{BK} through geometric techniques. After that, the singular and nonintegral cases were solved by Soergel \cite{S1, S2}. The proof of the Jantzen Conjecture, however, requires an extension of the geometric methods. The main idea was formulated in the early 1980s by Beilinson and Bernstein, while a complete proof (for regular case) was published in 1993 \cite{BB2}. In this paper, we find a method to translate the regular Jantzen filtration into a singular one (Proposition \ref{cprop1}). Thus the full Jantzen Conjecture (Theorem \ref{thm2}) follows from \cite{BB2} and the rigidity of Verma modules.

It is natural to expect an algebraic approach to these purely algebraic problems. In 1990, Soergel reformulated the Kazhdan-Lusztig Conjecture by using some special modules, so called Soergel modules \cite{S2}. These modules can be obtained from the Soergel bimodules \cite{EMTW} introduced by Soergel \cite{S3}. Soergel showed that the Kazhdan-Lusztig Conjecture holds if and only if the indecomposable Soergel bimodules are categorification of the Kazhdan-Lusztig basis of the Hecke algebra. In 2007, He also proved that such a theory can be generalized to any Coxeter system \cite{S4}. The analogue of the Kazhdan-Lusztig Conjecture in these settings, which is now konwn as Soergel Conjecture, was eventually settled by Elias and Williamson in 2014 \cite{EW}. Lately, Williamson also got an algebraic proof of the regular Jantzen Conjecture by giving the local hard Lefschetz theorem for Soergel bimodules \cite{W}. Combined with this, the full Jantzen Conjecture can be obtained by algebraic methods since the approach in this paper is purely algebraic.

The Jantzen Conjecture also plays a significant role in the classification of unitary Harish-Chandra modules. It is known that parabolic Verma modules appear as intermediate modules in the cohomological induction \cite{KV}, which can be used to constructed all the simple Harish-Chandra modules of a real reductive Lie group. The signature of invariant Hermitian forms on these intermediate modules may be used to calculate the signature of forms associated with the cohomologically induced modules and thus determine their unitarity.

The application of Jantzen filtration for computing these signatures was initiated by Vogan \cite{V}. Yee extended this idea in the setting of highest weight modules. In \cite{Y1}, the signed Kazhdan-Lusztig polynomials are defined to give signatures of forms on regular highest weight modules, while such polynomials were shown equal to classical Kazhdan-Lusztig polynomials evaluated at $-q$ and multiplied by a sign in \cite{Y2}. The proof depends heavily on the validity of regular Jantzen conjecture. It is expected that results in this paper could shed some light on the singular case of related problem, which have attracted more attention since the irreducible unitary Harish-Chandra modules with strongly regular infinitesimal characters are already known \cite{SR}. In \cite{ALTV}, a general algorithm was built to find all the unitary representations, where usage of regular Jantzen Conjecture is still an essential step. The result of the singular case might be used to calculate unitary representations more efficiently.

The Jantzen filtration can also be defined in various setting such as quantum groups \cite{APW} and Weyl modules \cite{JM} and superalgebras \cite{SZ}. It is natural to consider similar problems in those settings.

The paper is organized as follows. The Jantzen filtration is defined in Section 2. In section 3, we solve a problem raised by Jantzen which states that the translation functor sends Jantzen filtrations of regular Verma modules to the filtrations of singular ones. In section 4, we show that Jantzen filtrations and radical filtrations are the same for Verma modules. This immediately yields the Jantzen Conjecture for singular infinitesimal characters. In section 5, The coincidence between Jantzen filtrations and radical filtrations is extended to the case of parabolic Verma modules.

I am very grateful to Peng Shan for explaining her results about Jantzen filtrations to me. I would also like to thank Wai Ling Yee for her insightful comments.

%
%
\section{Preliminaries}
%
%


The Jantzen filtration for highest weight modules will be defined at the end of this section, after necessary notation are introduced. Most results in this section can be found in \cite{H3} or \cite{GJ}.

\subsection{Complex Lie algebras}

Recall that $\frg$ is a complex semisimple Lie algebra with a fixed Borel subalgebra $\frb=\frh\oplus\frn$, where $\frh$ is a Cartan subalgebra and $\frn$ the nilradical of $\frb$. Set $Q:=\bbZ\Phi$, where $\Phi$ is the root system of $(\frg, \frh)$ with a simple system $\Delta$ corresponding to $\frb$. Choose a basis $E_\beta$ ($\beta\in\Phi$), $H_\alpha$ ($\alpha\in\Delta$) of $\frg$ so that $E_\beta$ is a nonzero vector in the root space corresponding to $\beta$ and $[E_\alpha, E_{-\alpha}]=H_\alpha$. There is an anti-involution $\sigma$ on $\frg$ which fixes $h\in\frh$ and interchanges $E_\beta$ with $E_{-\beta}$ for $\beta\in\Phi$. Set $\overline\frn=\sigma(\frn)$. Then $\frg=\overline\frn\oplus\frh\oplus\frn$. Denote by $W$ the Weyl group of $\Phi$. The action of $W$ on $\frh^*$ is given by $s_\alpha\lambda=\lambda-\langle\lambda, \alpha^\vee\rangle\alpha$ for $\alpha\in\Phi$ and $\lambda\in\frh^*$. A weight $\lambda\in\frh^*$ is called \emph{dominant} (resp. \emph{anti-dominant}) if $\langle\lambda, \alpha^\vee\rangle\not\in\bbZ^{<0}$ (resp. $\langle\lambda, \alpha^\vee\rangle\not\in\bbZ^{>0}$) for all $\alpha\in\Delta$.

\smallskip


\subsection{Extended characters} Let $A=\bbC[t]_{(t)}$ for an indeterminate $t$. Set $\frg_A:=A\otimes_\bbC\frg$ and similarly define $\frb_A$, $\frh_A$, $\frn_A$, $\overline\frn_A$, .... Note that each $\lambda\in\frh_A^*$ determines an $\frh_A$-module $A_\lambda$ by $h\cdot a=\lambda(h)a$ for $a\in A$, which can be viewed as a $\frb_A$-module with trivial $\frn_A$ action. The \emph{Verma module} (over $A$) is defined by
\[
M(\lambda)_A:=U(\frg_A)\otimes_{U(\frb_A)} A_{\lambda-\rho}.
\]
The modules $M(\lambda)_A$ possess a character $\chi_\lambda$, which is an algebra homomorphism from the center $Z(\frg_A)$ of $U(\frg_A)$ to $A$ so that $z\cdot v=\chi_\lambda(z)v$ for all $z\in Z(\frg_A)$ and $v\in M(\lambda)_A$. Moreover, $\chi_\lambda=\chi_\mu$ when $\mu\in W\lambda$.

%

\subsection{The category $\caK_C$} The above Verma modules can be well studied in some special subcategories of $\frg_A$-modules. If $M$ is an $\frh_A$-module, for $\lambda\in\frh_A^*$, set
\[
M_\lambda:=\{v\in M\mid h\cdot v=\lambda(h)v,\ \mbox{for all}\ h\in\frh_A\}.
\]
Fix $\mu\in\frh_A^*$. Put $C:=\mu+Q\subset\frh_A^*$. Denote by $\caK_C$ the full subcategory of $\frg_A$-modules $M$ such that:

\smallskip
\begin{itemize}
  \item [(1)] $M=\sum_{\nu\in C} M_\nu$.

\smallskip

  \item [(2)] $U(\frn_A)v$ is a finitely generated $A$-module for any $v\in M$.

\smallskip
  \item [(3)] $M$ is finitely generated as a $\frg_A$-module.
\end{itemize}
\smallskip
Denote by $\lambda\ra\overline\lambda$ the canonical map $\frh_A^*\ra\frh_A^*/t\frh_A^*\simeq\frh^*$. We call $D\subset C$ a \emph{block} if $\overline D=\{\bar\nu\mid\nu\in D\}$ is a nonempty intersection of $\overline C=\{\bar\nu\mid\nu\in C\}$ with a $W$-orbit in $\frh^*$. In other words, $\overline D=\overline C\cap W\overline\lambda$ for a $\lambda\in C$. Set
\[
\caJ_D:=\cap_{\lambda\in D}\Ker\,\chi_\lambda.
\]
For $M\in\caK_C$, define the submodule
\[
M_D:=\{v\in M\mid \forall\, x\in \caJ_D, \exists\, n\in\bbZ^{\geq0}\ \mbox{such that}\ x^n\cdot v=0\}.
\]

\begin{prop}[{\cite[Proposition 1.8.4]{GJ}}]\label{prop1}
Let $M\in\caK_C$. Then $M=\bigoplus_{D}M_D$, where the sum is taken over all blocks $D\subset C$. In particular, $M_D$ is nonzero for only finitely many blocks $D$.
\end{prop}

Let $\mathrm{pr}_D: M\ra M_D$ be the natural projection on $\caK_C$, that is, $\mathrm{pr}_D M=M_D$. Evidently $\mathrm{pr}_D$ is exact.

\subsection{Contravariant forms}

Note that the anti-involution $\sigma$ extends naturally to an anti-automorphism on the enveloping algebra $U(\frg_A)$.  Let $M$ be a $\frg_A$-module, we say a symmetric bilinear form $\caF(,)$ on $M$ is \emph{contravariant} if
\[
\caF(u\cdot v_1, v_2)=\caF(v_1, \sigma(u)\cdot v_2)
\]
for all $v_1, v_2\in M$ and $u\in U(\frg_A)$.

\begin{lemma}\label{Alem1}
Let $\caF$ be a contravariant form on $M\in\caK_C$. Then
\begin{itemize}
  \item [(1)] $\caF(M_\lambda, M_\mu)=0$ for $\lambda\neq\mu$ in $\frh_A^*$.
  \item [(2)] $\caF(M_{D_1}, M_{D_2})=0$ for blocks $D_1\neq D_2$ in $C$.
  \item [(3)] For $i=1, 2$, suppose that $M_i$ is a $\frg_A$-module with contravariant form $\caF_i$. Then $\caF=\caF_1\otimes \caF_2$ is a contravariant form on $M_1\otimes M_2$.
\end{itemize}
\end{lemma}
\begin{proof}
The lemma follows as in the case when $A$ is the field $\bbC$ (e.g., \cite[Proposition 3.14]{H3}, keeping in mind of \cite[1.8]{GJ}).
\end{proof}

Note that each $\lambda\in\frh_A^*$ induces a homomorphism $\lambda: U(\frh_A)\ra A$. Composing $\lambda$ with the natural projection $U(\frg_A)\simeq U(\overline\frn_A)\otimes U(\frh_A)\otimes U(\frn_A)\ra U(\frh_A)$, we get a linear map $\vf_\lambda: U(\frg_A)\ra A$.

If $M$ is a highest weight $\frg_A$-module generated by a highest weight vector $v^+$, we can construct contravariant forms $\caF$ on $M$ with
\[
\caF(u\cdot v^+, u'\cdot v^+)=\caF(v^+, \vf_\lambda(\sigma(u)u')v^+).
\]
The form is called \emph{canonical} if $\caF(v^+, v^+)=1$. The following lemma shows that the canonical form is determined by $v^+$.

\begin{lemma}\label{Alem2}
Let $M$ be a highest weight $\frg_A$-module. There is a unique contravariant form $\caF_M$ on $M$ determined by $\caF_M(v^+, v^+)=1$. Moreover, if $\caF'$ is another contravariant form on $M$, then $\caF=a\caF_M$ for some $a\in A$.
\end{lemma}

\subsection{Jantzen filtration} Now we take $C=\mu+\delta t+Q$ with $\mu\in\frh^*$ anti-dominant and $\delta\in\frh^*$ regular.

\begin{definition}\label{defj}
Let $M\in\caK_C$. Fix a contravariant form $\caF$ on $M$. Define
\[
M_\caF^j:=\{v\in M\mid \caF(v, M)\in (t^j)\}
\]
for $j\in\bbN$. Then $M_\caF^j$ is a filtration of $\frg_A$-modules. It determines a filtration of $\frg$-module $\overline{M}:=M/tM$ by $\overline M_\caF^j=M_\caF^j/(tM\cap M_\caF^j)$.  In particular, if $M$ is a highest weight module and $\caF$ is a canonical form on $M$, then we drop the subscript and call $\overline M^j$ the \emph{Jantzen filtration} of $\overline{M}$.
\end{definition}

\begin{remark}
	For $\lambda\in\frh^*$, the Jantzen filtration $M(\lambda)^j=\overline{M(\lambda+\delta t)_A}^j$ of the Verma module $M(\lambda)$ is defined this way. As shown in \cite[8.4]{W}, it is possible to get filtrations of $M(\lambda)$ with non semisimple layers when the deformation $\delta$ is not dominant. In the case of Verma modules, we will always assume that $\delta=\rho$ as well as Jantzen's original statement \cite{J2}.
\end{remark}

%
%
\section{Jantzen filtration under translation}
%
%


This section is devoted to proving a result (Proposition \ref{cprop1}) proposed by Jantzen \cite[5.17]{J2}. Let $A=\bbC[t]_{(t)}$ and $C=\mu+\delta t+Q$ with $\mu\in\frh^*$ anti-dominant and $\delta=\rho$. Suppose that $\lambda\in\mu+Q$ is a regular anti-dominant weight. Let $T_\lambda^\mu$ be the usual translation functor (see \cite[7.1]{H3}). Define $$\Phi_{[\mu]}:=\{\alpha\in\Phi\mid\langle\mu, \alpha^\vee\rangle\in\bbZ\}$$ and $$W_{[\mu]}:=\{w\in W\mid w\mu-\mu\in Q\}.$$ Then $\Phi_{[\mu]}$ is a root system with Weyl group $W_{[\mu]}$ (see \cite{B} or \cite{H3}). Let $\Delta_{[\mu]}$ be the simple system corresponding to $\Phi_{[\mu]}^+:=\Phi_{[\mu]}\cap\Phi^+$. Put $J=\{\alpha\in\Delta_{[\mu]}\mid\langle\mu, \alpha^\vee\rangle=0\}$. Denote by $W_J$ the Weyl group generated by reflections $s_\alpha$ with $\alpha\in J$. Let $\ell(,)$ (resp. $\ell_{[\mu]}(,)$) be the \textit{length function} on $W$ (resp. $W_{[\mu]}$). Obviously $\ell=\ell_{[\mu]}$ when $\mu$ is integral.

Set
\begin{equation*}\label{ceq1}
	W_{[\mu]}^J:=\{w\in W_{[\mu]}\mid \ell_{[\mu]}(ws_\alpha)=\ell_{[\mu]}(w)+1\ \mbox{for all}\ \alpha\in J\}.
\end{equation*}
Thus $W_{[\mu]}^J$ contains all the shortest representatives of the coset $W_{[\mu]}/W_J$. For any $\nu\in\frh^*$, there exists a unique anti-dominant weight $\mu$ and $w\in W_{[\mu]}^J$ such that $\nu=w\mu$ (see \cite[Proposition 3.5]{H3}).

\begin{prop}\label{cprop1}
Fix $w\in W_{[\mu]}$. There exists $k\in\bbN$ such that
\begin{equation}\label{cprop1eq1}
T_\lambda^\mu M(w\lambda)^j\simeq M(w\mu)^{j-k}
\end{equation}
for all $j\in\bbN$. In particular, if $w\in W_{[\mu]}^J$, then $k=0$.
\end{prop}

Here $M(w\mu)^{j-k}=M(w\mu)$ when $j<k$. For $\nu\in\frh^*$, denote by $L(\nu)$ the unique simple quotient of $M(\nu)$. The following result is quite useful.


\begin{lemma}[{\cite[Theorem 7.9]{H3}}]\label{cthm1}
Fix $w\in W_{[\mu]}$. Then $T_\lambda^\mu L(w\lambda)=L(w\mu)$ for $w\in W_{[\mu]}^J$ and is vanished otherwise.
\end{lemma}

\subsection{Extended translation functor} To prove Proposition \ref{cprop1}, we need to extend the definition of translation functors. Let $E(\mu-\lambda)$ be the finite diminsional simple $\frg$-module with extremal weight $\mu-\lambda$. Then $E(\mu-\lambda)_A:=A\otimes_\bbC E(\mu-\lambda)$ is a $\frg_A$-module of finite rank. Set $D=W_{[\mu]}\lambda+\delta t$ and $D'=W_{[\mu]}\mu+\delta t$. Define the (extended) \emph{translation functor} $\caT^\mu_\lambda$ on $M\in\caK_C$:
\begin{equation}\label{ceq2}
\caT^\mu_\lambda M=\mathrm{pr}_{D'}(E(\mu-\lambda)_A\otimes_A \mathrm{pr}_{D}M).
\end{equation}
Obviously $T_\lambda^\mu \overline M\simeq \caT^\mu_\lambda M/t\caT^\mu_\lambda M$ as $\frg$-modules.



\begin{prop}\label{cprop2}
For $w\in W_{[\mu]}$,
\[
\caT_\lambda^\mu M(w\lambda+\delta t)_A\simeq M(w\mu+\delta t)_A.
\]
\end{prop}
\begin{proof}
Take $M=M(w\lambda+\delta t)_A$ for $w\in W_{[\mu]}$. Then $M\in\caK_C$ and $\mathrm{pr}_D M=M_D=M$. The module $E(\mu-\lambda)_A\otimes_A M(w\lambda+\delta t)_A$ has a filtration with quotients which isomorphic to $M(w\lambda+\delta t+\nu')_A$, where $\nu'\in\frh^*$ are weights of $E(\mu-\lambda)$, counting multiplicities. By \cite[\S2.9]{J2} (see also \cite[Lemma 7.5]{H3}), $\mathrm{pr}_{D'}(M(w\lambda+\delta t+\nu')_A)$ is not vanished if and only if $v'=w\mu-w\lambda$. The lemma then follows from the exactness of $\mathrm{pr}_{D'}$.
\end{proof}


\subsection{Filtrations on tensor product and primary decomposition} In this subsection, we give two necessary lemmas for proving Proposition \ref{cprop1}. For simplicity, put $E:=E(\mu-\lambda)_A$ in this subsection. Recall that (see Definition \ref{defj}) we write $M^j:=M_\caF^j$ and $\overline M^j:=\overline M_\caF^j$ if $M$ is a highest weight module and $\caF$ is a canonical form on $M$.

\begin{lemma}\label{clem1}
Let $M\in\caK_C$ be a highest weight $\frg_A$-module. Denote by $\caF$ the contravariant form on $E\otimes_A M$ given by the tensor product of canonical forms on $E$ and $M$. Then
\[
(E\otimes_A M)_\caF^j=E\otimes_A M^j
\]
for $j\in\bbN$.
\end{lemma}
\begin{proof}
Note that the $\frg$-module $E(\mu-\lambda)$ admits a nondegenerate contravariant form over $\bbC$ \cite[Theorem 3.15]{H3}. This form can be extended to a canonical form $\caF_E$ on $E$. Let $\caF_M$ be a canonical form on $M$. Thus $\caF=\caF_E\otimes \caF_M$. The inclusion $E\otimes_A M^j\subset (E\otimes_A M)_\caF^j$ follows immediately from the definition. For the other direction, take a basis $\{e_1, \cdots, e_n\}$ of $E(\mu-\lambda)$ over $\bbC$ and its dual basis $\{e_1^\vee, \cdots, e_n^\vee\}$ with respect to the nondegenerate contravariant form on $E(\mu-\lambda)$. Then $\caF_E(e_i, e_l^\vee)=\delta_{i, l}$. Any $u\in E\otimes_A M$ can be written as $\sum_{i=1}^n e_i\otimes_A v_i$ with $v_i\in M$. Fix an integer $1\leq k\leq n$. If $u\in(E\otimes_A M)_\caF^j$, it follows from
\[
\caF_M(v_k, v)=\caF(\sum_{i=1}^n e_i\otimes v_i, e_k^\vee\otimes v)=\caF(u, e_k^\vee\otimes v)\in (t^j)
\]
for all $v\in M$ that $v_k\in M^j$. Thus $u\in E\otimes_A M^j$.
\end{proof}

\begin{lemma}\label{clem2}
Let $\caF$ be a nondegenerate contravariant form on $M\in\caK_C$. Then its restriction $\caF'$ on $M_D$ is also nondegenerate for any block $D\subset C$. Moreover,
\[
(M_\caF^j)_D=M_\caF^j\cap M_D=(M_D)_{\caF'}^j
\]
for any $j\in\bbN$. In particular, if $M_D$ is a highest weight module, then there exists $k\in\bbN$ such that $(M_\caF^j)_D=M_D^{j-k}$.
\end{lemma}


\begin{proof}
With Lemma \ref{Alem1}, it suffices to prove the last statement. Suppose that $M_D$ is a highest weight module. In view of Lemma \ref{Alem2}, $\caF'=a\caF''$ for some $a\in A$, where $\caF''$ is a canonical form on $M_D$. Since $A$ is a principal domain, we can find $k\in\bbN$ such that $a\in (t^k)$ and $a\not\in (t^{k+1})$. It follows that $(M_D)_{\caF'}^j=(M_D)_{\caF''}^{j-k}=M_D^{j-k}$.
\end{proof}

\subsection{Proof of Proposition \ref{cprop1}} Take $M=M(w\lambda+\delta t)_A$ for $w\in W_{[\mu]}$. Set $N=E\otimes_AM$ and $\caF=\caF_E\otimes\caF_M$. Proposition \ref{cprop2} implies $N_{D'}=M(w\mu+\delta t)_A$ is a highest weight module. It follows from Lemma \ref{clem1} and \ref{clem2} that
\[
\caT_\lambda^\mu M^j\simeq (E\otimes_A M^j)_{D'}\simeq (N_\caF^j)_{D'}\simeq N_{\caF}^j\cap N_{D'}\simeq N_{D'}^{j-k}
\]
for a $k\in\bbN$. We get $\caT_\lambda^\mu M(w\lambda+\delta t)^j_A\simeq M(w\mu+\delta t)_A^{j-k}$. Note that $\overline{M(w\lambda+\delta t)_A}=M(w\lambda)$ and $\overline{M(w\mu+\delta t)_A}=M(w\mu)$. This immediately yields $T_\lambda^\mu M(w\lambda)^j\simeq M(w\mu)^{j-k}$. Now suppose $w\in W_{[\mu]}^J$. If $k>0$, one has $T_\lambda^\mu M(w\lambda)=M(w\mu)$ and $T_\lambda^\mu M(w\lambda)^k\simeq M(w\mu)$ by taking $j=0, k$. This forces $T_\lambda^\mu(M(w\lambda)/M(w\lambda)^k)=0$ and $T_\lambda^\mu L(w\lambda)=0$. By Lemma \ref{cthm1}, we arrive at a contradiction $L(w\mu)=0$.


%
%
\section{Jantzen Conjecture of Verma modules}
%
%
In this section, we will prove the Jantzen Conjecture for singular infinitesimal characters.

We say a filtration of $M\in\caO$ is a \textit{Loewy filtration} if it has the shortest possible length provided its successive quotients is semisimple. The length of such a filtration is called the \textit{Loewy length} of $M$. The radical filtration and socle filtration are two extreme examples of Loewy filtrations. Denote by $\Rad^{i}M$ (resp. $\Soc^{i}M$) the $i$-th radical (resp. socle) filtration of $M$ when $i\geq0$. If $M=M^0\supset M^1\supset\cdots$ is a descending chain of Loewy filtration, then $\Rad^{i}M\subset M^i\subset \Soc^{s-i}M$, where $s$ is the Loewy length of $M$. We say $M$ is \emph{rigid} if its radical and socle filtrations coincide (thus its Loewy filtration is unique). For convenience, set $\Rad^{i}M=M$ and $\Soc^iM=0$ when $i\in\bbZ^{<0}$. If $M$ is rigid, then $\Rad^{i}M=\Soc^{s-i}M$ for $i\in\bbZ$. Put $\Rad_{i}M=\Rad^{i}M/\Rad^{i+1}M$ and $\Soc_{i}M=\Soc^{i}M/\Soc^{i-1}M$.

As in the previous section, the weight $\mu$ is anti-dominant and $\lambda\in\mu+Q$ is regular anti-dominant. Recall that $J=\{\alpha\in\Delta_{[\mu]}\mid\langle\mu, \alpha^\vee\rangle=0\}$. Any $w\in W_{[\mu]}$ can be uniquely written as $w=yx$ with $y\in W_{[\mu]}^J$ and $x\in W_J$.

\begin{prop}\label{vprop1}
Suppose that $w=yx\in W_{[\mu]}$ with $y\in W_{[\mu]}^J$ and $x\in W_J$. Then
\begin{itemize}
\item [(1)] $M(w\mu)$ is rigid with Loewy length $\ell_{[\mu]}(y)+1$.
\item [(2)] $T_\lambda^\mu \Rad^{j} M(w\lambda)=\Rad^{j-\ell_{[\mu]}(x)} M(y\mu)$ for any $j\in\bbN$.
\end{itemize}
\end{prop}
\begin{proof}
With $\lambda\in\mu+Q$, we can set $\ell':=\ell_{[\lambda]}=\ell_{[\mu]}$. So $\ell'(w)=\ell'(y)+\ell'(x)$.

(1) Since $w\mu=y\mu$, it suffices to show that $M(y\mu)$ is rigid with Loewy length $\ell'(y)+1$. If $\mu$ is integral and regular, this is a consequence of the Kazhdan-Lusztig theory \cite[Theorem 8.15]{H3}. If $\mu$ is integral and singular, the result follows from \cite[Theorem 1.4.2]{I2}. If $\mu$ is not integral, the proof can be reduced to the integral case by Soergel's category equivalence \cite[Theorem 11]{S2}.

(2) Since $\lambda$ is regular, $W_{[\lambda]}^{\emptyset}=W_{[\lambda]}=W_{[\mu]}$. The Loewy length of $M(w\lambda)$ (resp. $M(y\lambda)$) is $\ell'(w)+1$ (resp. $\ell'(y)+1$) by (1). Note that $M(y\lambda)\subset M(w\lambda)$. Thus
\begin{equation}\label{vprop1eq1}
M(y\lambda)\subset\Soc^{\ell'(y)+1}M(w\lambda)=\Rad^{\ell'(x)}M(w\lambda)\subset M(w\lambda).
\end{equation}
In view of \cite[Theorem 7.6]{H3}, we have $T_\lambda^\mu M(w\lambda)=M(w\mu)=M(y\mu)=T_\lambda^\mu M(y\lambda)$. Thus $T_\lambda^\mu(M(w\lambda)/M(y\lambda))=0$ and $T_\lambda^\mu\Rad^{\ell'(x)}M(w\lambda)=M(w\mu)$. If $i<\ell'(x)$, then $\Rad_{i}M(w\lambda)$ is a subquotient of $M(w\lambda)/M(y\lambda)$ by (\ref{vprop1eq1}) and $T_\lambda^\mu\Rad_{i}M(w\lambda)=0$. We claim that
\[
M(w\mu)=T_\lambda^\mu\Rad^{\ell'(x)}M(w\lambda)\supset\cdots\supset T_\lambda^\mu\Rad^{\ell'(w)+1}M(w\lambda)=0
\]
is a Loewy filtration of $M(w\mu)$. In fact, the $j$-th quotient of this filtration is equal to $T_\lambda^\mu\Rad_{j+\ell'(x)}M(w\lambda)$, which is semisimple by Lemma \ref{cthm1}. The length of this filtration is equal to $\ell'(w)+1-\ell'(x)=\ell'(y)+1$, which is the Loewy length of $M(w\mu)$ in view of (1). Hence the rigidity of $M(w\mu)$ forces $T_\lambda^\mu \Rad^{j+\ell'(x)} M(w\lambda)=\Rad^{j} M(w\mu)$ for $0\leq j\leq \ell'(y)+1$.
\end{proof}

\begin{remark}
If $\mu$ is integral, Proposition \ref{vprop1}(2) is obtained in \cite[Theorem 2.3.2]{I2}.
\end{remark}

\begin{theorem}\label{thm1}
For $w\in W_{[\mu]}^J$ and $j\in\bbN$, we have
\begin{equation*}\label{vpeq1}
M(w\mu)^j\simeq \Rad^jM(w\mu).
\end{equation*}
\end{theorem}
\begin{proof}
Proposition \ref{cprop1} implies $T_\lambda^\mu M(w\lambda)^j\simeq M(w\mu)^{j}$. By \cite{BB2}, one has $M(w\lambda)^j=\Rad^j M(w\lambda)$. The theorem then follows from Proposition \ref{vprop1}.
\end{proof}

\begin{theorem}\label{thm2}
The Jantzen Conjecture holds for all infinitesimal characters.
\end{theorem}

\begin{proof}
It suffices to consider $M(x\mu)\subset M(w\mu)$ for $x, w\in W_{[\mu]}^J$. With Theorem \ref{thm1}, it suffices to show that $M(x\mu)\cap \Rad^jM(w\mu)=\Rad^{j-r}M(x\mu)$
for $j\in\bbN$ and $r=|\Phi^+_{w\mu}|-|\Phi^+_{x\mu}|$. By Soergel's category equivalence \cite{S2}, we can assume that $\mu$ is integral. In this case, $W_{[\mu]}=W$. Recall that $\Phi_{w\mu}^+=\{\alpha>0\mid \langle w\mu, \alpha^\vee\rangle\in\bbZ^{>0}\}$. So
\[
\Phi_{w\mu}^+=\{\alpha>0\mid \langle \mu, (w^{-1}\alpha)^\vee\rangle\in\bbZ^{>0}\}=\{\alpha>0\mid w^{-1}\alpha<0\}=\ell(w).
\]
Thus $r=\ell(w)-\ell(x)$. By Proposition \ref{vprop1}, the Verma module $M(w\mu)$ (resp. $M(x\mu)$) is rigid with Loewy length $\ell(w)+1$ (resp. $\ell(x)+1$). Therefore
\[
\begin{aligned}
M(x\mu)\cap \Rad^jM(w\mu)=&M(x\mu)\cap \Soc^{\ell(w)+1-j}M(w\mu)\\
=&\Soc^{\ell(w)+1-j}M(x\mu)\\
=&\Rad^{j-(\ell(w)-\ell(x))}M(x\mu).
\end{aligned}
\]

\end{proof}

%
%
\section{Jantzen filtrations of parabolic Verma modules}
%
%

Since homomorphisms between parabolic Verma modules are usually not embeddings, we can not build result like Theorem \ref{thm2} for parabolic Verma modules. Fortunately, the Jantzen filtrations and radical filtrations still coincide. This can be obtained by a similar argument with relatively cumbersome notation.

\subsection{Parabolic Verma modules} Fix $I\subset\Delta$. It generates a subsystem $\Phi_I\subset\Phi$ with a positive root system $\Phi^+_I:=\Phi_I\cap\Phi^+$. Denote by $W_I$ the Weyl group of $\Phi_I$, with the longest element $w_I$. Let $\frp=\frl\oplus\fru$ be the standard parabolic subalgebra of $\frg$ corresponding to $I$ with Levi subalgebra $\frl$ and nilradical $\fru$. The category $\caO^\frp$ is the full subcategory of $\caO$ consisting of modules which are locally $\frp$-finite \cite{H3}. Put
\[
\Lambda_I^+:=\{\lambda\in\frh^*\ |\ \langle\lambda, \alpha^\vee\rangle\in\bbZ^{>0}\ \mbox{for all}\ \alpha\in I\}.
\]
For any $\lambda\in\Lambda_I^+$, there exists a simple finite-dimensional $\frl$-module $F(\lambda-\rho)$ with highest weight $\lambda-\rho$. Extend it to a $\frp$-module with trivial $\fru$ action. The (ordinary) \emph{parabolic Verma module} is defined by
\[
M_I(\lambda):=U(\frg)\otimes_{U(\frp)} F(\lambda-\rho).
\]
The module $M_I(\lambda)$ can be parameterized by $\lambda=w_Iw\mu$ for an anti-dominant weight $\mu\in\frh^*$ and $w\in{}^IW_{[\mu]}^J$ (e.g., \cite{BN}), where $J=\{\alpha\in\Delta_{[\mu]}\mid\langle\mu, \alpha^\vee\rangle=0\}$ and
\[
{}^IW_{[\mu]}^J:=\{w\in W_{[\mu]}^J\mid \ell_{[\mu]}(s_\alpha w)=\ell_{[\mu]}(w)+1\ \mbox{and}\ s_\alpha w\in W_{[\mu]}^J\ \mbox{for all}\ \alpha\in I\}.
\]
Note that $I\subset\Delta_{[\mu]}$. Denote ${}^IW_{[\mu]}:={}^IW_{[\mu]}^\emptyset$. The set ${}^IW_{[\mu]}^J$ contains exactly those elements $w$ in ${}^IW_{[\mu]}\cap W_{[\mu]}^J$ (the set of the shortest $(W_I, W_J)$ double cosets representatives in $W_{[\mu]}$) such that $|W_IwW_J|=|W_I|\cdot|W_J|$. It is possibly empty, while ${}^IW_{[\mu]}\cap W_{[\mu]}^J$ always contains the identity element of $W_{[\mu]}$.

From now on, $\mu\in\frh^*$ is always anti-dominant and $\lambda\in\mu+Q$ is anti-dominant regular. Let $Z_\frp: \caO\ra\caO^\frp$ be the Zuckerman functor which takes the maximal quotient in $\caO^\frp$. It is the left adjoint of the inclusion functor
$i_\frp: \caO^\frp\ra\caO$ and commutes with any projective functor, including $T_\lambda^\mu$ (see \cite[Proposition 6.1]{M}). In particular, $Z_\frp(M(\lambda))=M_I(\lambda)$ when $\lambda\in\Lambda_I^+$ and $Z_\frp(M(\lambda))=0$ otherwise.

\begin{prop}\label{gprop1}
Let $w\in {}^IW_{[\mu]}$.
\[
T_\lambda^\mu \Rad^jM_I(w_Iw\lambda)\simeq\begin{cases}\Rad^{j-\ell_{[\mu]}(x)}M_I(w_Iy\mu), &\text{if $w=yx$, $y\in {}^IW_{[\mu]}^J, x\in W_J$};\\
0, &\text{otherwise}.
\end{cases}
\]
\end{prop}
\begin{proof}
The case $j=0$ follows from $T_\lambda^\mu Z_\frp=Z_\frp T_\lambda^\mu$ and Proposition \ref{vprop1}. Indeed, set $\ell':=\ell_{[\mu]}$. Then
\[
T_\lambda^\mu M_I(w_Iw\lambda)=T_\lambda^\mu Z_\frp(M(w_Iw\lambda))=Z_\frp T_\lambda^\mu(M(w_Iw\lambda))=M_I(w_Iy\mu).
\]
If $j>0$, it suffices to consider the case $w=yx$ with $y\in {}^IW_{[\mu]}^J$ and $x\in W_J$.
Let $f: M(w_Iw\lambda)\ra M_I(w_Iw\lambda)$ be the natural projection. It follows from the exactness of $T_\lambda^\mu$ that $g=T_\lambda^\mu(f): M(w_Iw\mu)=M(w_Iy\mu)\ra M_I(w_Iy\mu)$ is also a projection. Combined with Proposition \ref{vprop1},
\[
\begin{aligned}
T_\lambda^\mu \Rad^jM_I(w_Iw\lambda)&\simeq T_\lambda^\mu(f(\Rad^jM(w_Iw\lambda)))\simeq T_\lambda^\mu(f)(T_\lambda^\mu \Rad^jM(w_Iw\lambda))\\
&\simeq g(\Rad^{j-\ell'(x)}M(w_Iy\mu))\\
&\simeq \Rad^{j-\ell'(x)}M_I(w_Iy\mu).
\end{aligned}
\]
\end{proof}

\begin{remark}
If $\mu$ is integral, the above result is equivalent to Theorem 4.3 in \cite{CM}. The above proposition is a natural generalization to the nonintegral case.
\end{remark}

\subsection{Jantzen filtrations of parabolic Verma modules} This subsection is devoted to proving the following result, which is a generalization of Proposition \ref{cprop1}.

\begin{prop}\label{gprop2}
Fix $w=yx\in {}^IW_{[\mu]}$ with $y\in {}^IW_{[\mu]}^J$ and $x\in W_J$. There exists a $k\in\bbN$ such that
\begin{equation}\label{gprop1eq1}
T_\lambda^\mu M_I(w_Iw\lambda)^j\simeq M_I(w_Iy\mu)^{j-k}
\end{equation}
for all $j\in\bbN$. In particular, if $w\in {}^IW_{[\mu]}^J$, then $k=0$.
\end{prop}

Similarly, we need to extend the definition of parabolic Verma modules. Suppose $A=\bbC[t]_{(t)}$. Choose the dominant weight $\delta\in\frh^*$ such that $\langle\delta, \alpha^\vee\rangle=0$ for $\alpha\in I$ and $\langle\delta, \alpha^\vee\rangle=1$ for $\alpha\in\Delta\backslash I$. Thus $A_{\delta t}$ can be viewed as an $\frl_A$-module with trivial $E_\alpha$ action for $\alpha\in\Phi_I$. For any $\nu\in\Lambda_I^+$, there exists a simple finite-dimensional $\frl$-module $F(\nu-\rho)$ with highest weight $\nu-\rho$. Denote $F(\nu-\rho+\delta t)_A=A_{\delta t}\otimes_\bbC F(\nu-\rho)$. Then $F(\nu-\rho+\delta t)_A$ is an $\frl_A$-module. Extend it to a $\frp_A$-module with trivial $\fru_A$-action. The \emph{parabolic Verma module} (over $A$) is defined by
\[
M_I(\nu+\delta t)_A:=U(\frg_A)\otimes_{U(\frp_A)} F(\nu-\rho+\delta t)_A.
\]

\begin{remark}
To avoid excessive notation, our definition of parabolic Verma modules over $A$, which is enough for our argument in this paper, is not as general as that of Verma modules over $A$.
\end{remark}

\begin{prop}\label{gprop3}
For $w\in {}^IW_{[\mu]}$,
\[
\caT_\lambda^\mu M_I(w_Iw\lambda+\delta t)_A\simeq\begin{cases}M_I(w_Iy\mu+\delta t)_A, &\text{if $w=yx$, $y\in {}^IW_{[\mu]}^J, x\in W_J$};\\
0, &\text{otherwise}.
\end{cases}
\]
\end{prop}

\begin{proof}
This is a generalization of Proposition \ref{cprop2}, we adopt the proof there with necessary modification. The module $E(\mu-\lambda)_A\otimes_A M_I(w_Iw\lambda+\delta t)_A$ has a filtration with quotients isomorphic to $M_I(\nu'+\rho+\delta t)_A$, where $\nu'$ are highest weights of simple $\frl$-modules appeared in the decomposition of $\frl$-module $E(\mu-\lambda)\otimes_\bbC F(w_Iw\lambda-\rho)$. In view of Proposition \ref{gprop1}, the module $\mathrm{pr}_{D'}(M_I(\nu'+\rho+\delta t)_A)$ is vanished unless $v'=w_Iw\mu-\rho$, with multiplicity $1$. The lemma then follows from the exactness of $\mathrm{pr}_{D'}$.
\end{proof}

\subsection{Proof of Proposition \ref{gprop2}} Imitate closely the proof of Proposition \ref{cprop1}, keeping in mind of Proposition \ref{gprop3}.

\subsection{Coincidence of Jantzen filtration and radical filtration}

\begin{theorem}\label{thm3}
For $w\in {}^IW_{[\mu]}^J$ and $j\in\bbN$, we have
\begin{equation}\label{t3eq1}
M_I(w_Iw\mu)^j\simeq \Rad^jM_I(w_Iw\mu).
\end{equation}
\end{theorem}
\begin{proof}
The proof is similar to that of Theorem \ref{thm1}. In view of Proposition \ref{gprop1} and \ref{gprop2}, we only need to show that Jantzen filtration and radical filtration of a regular (ordinary) parabolic Verma module coincide. Here we adopt Shan's idea \cite{Sh}, along with results in \cite{BB2, S2}. In \cite[Lemma 5.1]{Sh}, Shan proved the existence of regular functions on related schemes. As pointed out in \cite[Remark 5.2 and 6.7]{Sh}, this generalized the regular functions obtained in \cite[Lemma 3.5.1]{BB2} associated with Borel subalgebras of finite types to the case of parabolic subalgebras of affine types (and of finite types by a simpler argument).  The generalized regular functions can be used to define Jantzen filtration for regular parabolic Verma modules \cite[5.5]{Sh}. In the integral case, the Jantzen filtration can be interpreted as a weight filtration \cite[\S5.1.3]{BB2} which coincides with the radical filtration by \cite{BG, C}. For the nonintegral case, note that the Jantzen filtration of a parabolic Verma module is determined by a homomorphism from itself to its dual \cite[Definition 1.6]{Sh}. Such a homomorphism are well preserved when reduced to integral cases by Soergel's results \cite{S2}.
\end{proof}

\begin{remark}
	In \cite{HX}, we obtained a sum formula for radical filtrations of parabolic Verma modules. We can reproduce this formula by the Jantzen sum formula for parabolic Verma modules (see \cite[Remark 4.11]{HX}) and Theorem \ref{thm3}.
\end{remark}

\end{document}